\documentclass[11pt,tbtags,reqno]{article}
\usepackage[margin=1.0in]{geometry}
\usepackage{amssymb}
\usepackage{amsmath}
\usepackage{amsthm}
\usepackage[swedish, english]{babel}
\usepackage{psfrag}   
\usepackage{graphicx}
\usepackage[retainorgcmds]{IEEEtrantools}
\usepackage{bbm}
\usepackage[square, comma, numbers, sort&compress]{natbib}
\usepackage{enumerate}
\usepackage{soul}
\usepackage{hyperref}
\usepackage[colorinlistoftodos, textsize=tiny, backgroundcolor=yellow
, linecolor=green]{todonotes}
\usepackage{sidecap}
\usepackage{float}

\usepackage{bm}

\newcommand{\E}{\mathbb E}
\newcommand{\R}{\mathbb R}

\newcommand{\D}{\mathcal D}
\newcommand{\F}{\mathcal F}
\newcommand{\I}{\mathcal I}
\newcommand{\G}{\mathcal G}

\newcommand{\U}{\mathcal U}

\newcommand{\ud}{\mathrm{d}}
\newcommand{\vd}{\mathrm{d}}

\def\P{{\mathbb P}}

\newcommand{\Var}{\mathop{\mathrm{Var}}\nolimits}

\numberwithin{equation}{section}
\theoremstyle{plain}
\newenvironment{example}[1][Example]{\begin{trivlist}
\item[\hskip \labelsep {\bf Example}]}{\end{trivlist}}

\newtheorem{theorem}{Theorem}[section]

\newtheorem{proposition}[theorem]{Proposition}

\newtheorem{remark}[theorem]{Remark}
\theoremstyle{definition}
\theoremstyle{example}

\begin{document}

\title{A SEQUENTIAL ESTIMATION PROBLEM WITH CONTROL AND DISCRETIONARY STOPPING} 

\author{\textsc{Erik Ekstr\"om\,}\thanks{\,Department of Mathematics, Uppsala University, Box 256, 75105 Uppsala, Sweden \mbox{(email: 
{\it ekstrom@math.uu.se}).} Support from the Swedish Research Council under grant 2019-03525
is gratefully acknowledged.} \hspace{0.5mm}
and \textsc{ Ioannis Karatzas\,}\thanks{
\, Departments of Mathematics and Statistics, Columbia University, 2990 Broadway, New York, NY 10027, USA \mbox{(email: 
{\it ik1@columbia.edu})}.   
Support from the National Science Foundation under grant NSF-DMS-20-04977 is gratefully acknowledged.
}}

\maketitle

{\centering\small
{\it Dedicated to Professor Alain Bensoussan on the occasion of his 80th birthday.} \par}

\medskip

\begin{abstract}

\smallskip
\noindent
We show that ``full-bang" control is optimal in a problem which combines features of  (i) sequential least-squares  {\it estimation} with Bayesian updating, for a random quantity observed in a bath of white noise; (ii)  bounded  {\it control} of the rate at which observations are received, with a superquadratic cost per unit time; and (iii) ``fast" discretionary  {\it stopping}. We  develop also the optimal filtering and stopping rules in this context.

\end{abstract}




\medskip

\section{Introduction and Summary}
\label{sec1}

Consider    trying to estimate a quantity about which there is   uncertainty, and which   cannot be   observed  directly.   We have instead access to a stream of observations that this   quantity affects  and, based on this stream,    try to find an  estimator of the unobservable quantity which is ``optimal" in the sense of least-squares. Access to  the stream of observations is in   our control,  though costly: we can decide  at any given time $t$ the rate $u(t) \in (0,1]$ at which we receive it, but have to    pay  a   positive
cost $ h (u(t))$ per unit of time  for as long as we keep observing. We can choose also the termination time $\tau$ of the experiment. {\it How is this triple problem, of sequential estimation (filtering), optimal control, and stopping, to be    resolved in a way that  balances the conflicting requirements of fidelity in estimation  and of cost minimization?}  

\smallskip
We study here   a   stylized form of this question in a Bayesian setting. We assume that the unobservable quantity is a   random
variable $X$ with known  ``prior" distribution, and that we observe  sequentially 
 the   process 
 \begin{IEEEeqnarray}{rCl} 
 \label{E:Y}
 Y(t)=X \int^t_0 u(s)\,\ud s + W(t)\,, \qquad 0 \le t < \infty\,.
 \end{IEEEeqnarray} 
 
 \bigskip
 \newpage
 \noindent
  Here  $W(\cdot)= \big( W(t)\big)_{ 0 \le t < \infty}\, $   is a  standard Wiener process, independent of the random variable $X$;     
 we assume that  we know the    distribution $\,{\bm \mu}\,$ of $X$ and that it has finite, positive variance; and posit that, at
all times $t \in [0, \infty)$, 
 we have continual access to the      filtration  $\mathbb F=\big(\F(t)\big)_{0\leq t<\infty}$ generated 
by the ``observations process"
$Y(\cdot)$ in (\ref{E:Y}). The ``feedback control" process $u (\cdot) = ( u(t))_{ 0 \le t < \infty}\,$ is    adapted to this
filtration  $\mathbb F$, takes   values in $(0,1],$  and satisfies the non-degeneracy condition (\ref{2.1}) below. 
We can select also a   time $\tau$ for terminating the experiment,  in the collection $ \,\mathcal T $  of stopping times of  the 
filtration $\mathbb F$.

\smallskip
The objective then, is to find a pair $\, (\tau^*, u^* (\cdot))\,$ that minimizes the total expected cost 
\begin{equation}
\label{1.1}
\E   \Big(X-\widehat X (\tau) \Big)^2 + \,\E \int_0^\tau h \big(u(t)\big)\, \ud t  
\end{equation}
of estimation-plus-control, over pairs $\, (\tau, u (\cdot))\,$consisting of  stopping times and control policies; hopefully in a manner that leads also to a ``fastest possible" termination time. 
Here $\, \widehat X (t) = \E \big[ \, X \, \big| \, \F (t) \, \big]\,$ is the least-squares estimate of $X$ at any given time $t\in [0,\infty),$ given the observations up to that  time; and $  h   (\cdot) $ is a positive, continuous, non-decreasing function on (0,1], which measures the instantaneous cost of control and satisfies the ``super-quadratic" condition (\ref{2.8}) --- or, a bit more generally, the requirement (\ref{2.8.5}). 

{\it In such a context how ``bold", or how ``timid", should one be, when choosing the rate $u(t) \in (0,1]$ at which observations are obtained?}

\subsection{Preview}
\label{sec1.1}

 We offer a precise formulation for this problem   in  section \ref{sec2}, 
based on changes of probability measure (``weak-solution formulation")  and the \textsc{Girsanov} theorem.

 Elementary filtering theory, time-change techniques using   the \textsc{Dambis-Dubins-Schwarz} theorem, and a re-parametrization based on the ``posterior" (conditional) mean and variance, are then deployed in sections \ref{sec3},  \ref{sec4} to reduce the problem to manageable proportions --- and to describe in detail its optimal filtering and stopping rules corresponding to any given control. A bit more specifically, but still in very broad brushes and with tentative notation, the original problem is reduced to one based on the conditional (posterior) mean $\widehat X (t)$ and variance $V(t)$ across times $t \in [0,\infty)$, via the action of suitable functions $G$ and $H,$  $\Psi$ of time and space, which satisfy  suitable nonlinear partial differential equations of parabolic type; and via re-parametrization, based on the martingale $\widehat X (\cdot)$. The impact of the control is removed by deploying the time-change $\,A^u(\cdot):=\int_0^{\,\cdot} u^2(s)\,\ud s\,,$ which re-writes the first term in \eqref{1.1} as
 $$
\E   \Big(X-\widehat X (\tau)\Big)^2     =\,\E \big[ V(\tau) \big] = V(0) - \E \int_0^\tau V^2( t) \, \mathrm{d}  A^u (t) = V(0) - \E \int_0^{A^u (\tau)} \Psi^2 \big( s , Q^u (s) \big) \, \mathrm{d} s\,,
$$
   in terms of the  diffusion process   $Q^u (\cdot)$ whose distribution does not depend on the control $u(\cdot)$. The cumulative impact of observations is thus measured by $A^u(\tau)$, and the overall problem is cast as the minimization of 
$$
\E \left[\int_0^\tau\left(\frac{h(u(t))}{u^2(t)}-\Psi^2 \big(A^u(t),\widehat X (t) \big)\right) \ud A^u(t)\right]  
$$
over stopping times $\tau$ and controls $u(\cdot)$.   Finally, we show in section \ref{sec5}     that, for     cost functions $h(\cdot)$ satisfying the super-quadratic condition $h(u) / u^2 \ge h(1), ~ 0 <u \le 1$ of (\ref{2.8}), control   of ``bold play" or ``full-bang" type $\, u^* (\cdot) \equiv 1\,$ is not only optimal,  but leads also to an  optimal  termination time $\tau^*$ which is ``fastest possible": 
if $(\tau, u(\cdot))$ is another optimal pair, then 
we have the stochastic dominance
\[
\P^*(\tau^*>t) \, \leq \, \P^{\,    u }\, \big(  \tau>t\big), \qquad 0\leq t<\infty\,.
\]
In this manner, we end up with the sequential estimation problem for $X$ from observations $Y(t) = X \cdot t + W(t),~ 0 \le t < \infty\,, $ treated in detail in  \cite{EKV} and admitting  explicit solutions for  \textsc{Gauss} and \textsc{Bernoulli} prior distributions   on $X$. These are discussed in section \ref{sec6} in our present context.

The super-quadratic 
condition   (\ref{2.8}) posits that the cost rate $h \big( u(t)\big)$ of deploying control $u(t) \in (0,1]$ at time $t$, measured relative to the local rate $\frac{\ud~}{\ud t} A^u (t)= u^2 (t)$ of data acquisition, is minimal when $u(t)=1$. This condition guarantees that the optimal stopping aspect of this problem is not trivial, i.e., that we do not end up observing {\it ad infinitum} by selecting $\tau = \infty$.

\subsection{Related Work}
\label{sec1.2}

The results in this paper provide  a rare study of   problems which  combine all three features of optimal filtering, stopping and control,  yet admit  fairly explicit answers.
Related studies,   again with fairly explicit answers,  are those by \textsc{Dalang  \& Shiryaev} \cite{DS} in the context of detecting a change-point, and by \textsc{Harrison \& Sunar} \cite{HS}  in the context of investment timing with incomplete information.
 Whereas, some general theory  for such problems involving all these three features, is developed on the last pages (Chapter 4, Section 6) of  \textsc{Bensoussan \& Lions}  \cite{BL}.  A more recent contribution is \cite{ELO}, where a fraud detection game is studied.

\smallskip
Problems involving combined control and stopping  
have been studied quite extensively, starting with the  ``leavable gambling houses" of \textsc{Dubins  \& Savage} \cite{DS1}. 
They arise, for instance, in target-tracking, where one has to stay in the vicinity of a target by spending fuel, declare when one has arrived ``sufficiently close”, then decide whether to engage the target or not; in portfolio optimization with horizon chosen by the investor; and in American option valuation under constraints. The monographs  \cite{Kry}, \cite{ElK}, \cite{BL}, \cite{FS} and the papers \cite{L}, \cite{M}, \cite{KZ} contain  general theoretical results, based on partial differential equation and/or probabilistic methods. There is also a host of specific problems, of this combined control-and-stopping type,  that admit explicit solutions: we mention \cite{DZ}, \cite{KS3}, \cite{KS4}, \cite{KW}, \cite{KOWZ}, \cite{KM} and the references cited there, as representative examples. 

\smallskip
The least-squares estimation error used in \eqref{1.1} is very special, in that it affords a direct link to the posterior variance which makes explicit computation possible. It would be   interesting to see how far an analysis along the lines of the one carried out here can go, using a different criterion for the estimation error; in particular, a one-sided criterion such as $\,\E \big[ \big(X-\widehat X (\tau)\big)^+ \, \big] ,$ in the manner carried out in  \cite{K} in the context of a detection problem.


\section{The Model}
\label{sec2}

Consider a probability space $(\Omega, \mathcal A,\P)$ rich enough to accommodate a standard, scalar Brownian motion $Y(\cdot)=(Y(t))_{0\leq t<\infty}$ and an independent random variable $X$.
This quantity has known distribution ${\bm \mu}$ with positive and finite variance 
$$
0< V(0)\equiv   \Var(X):= \int_\R   b^2   \, {\bm \mu} (\ud b) - \left(\int_\R   b^2   \, {\bm \mu} (\ud b) \right)^2 \le \int_\R   b^2   \, {\bm \mu} (\ud b)< \infty.
$$ 
For technical convenience, we assume also 
$$
\int_\R \exp \big\{ \alpha \, b^2 \big\} \, {\bm \mu} (\ud b) < \infty\,, \hbox{    ~~for some  }~ \alpha \in (0, \infty)\,.
$$
We denote by $\mathbb F=\big(\F(t)\big)_{0\leq t<\infty}$ (respectively, by $\mathbb G=\big(\mathcal G(t)\big)_{0\leq t<\infty}\,$) the smallest right-continuous filtration to which the process $Y(\cdot)$
(resp., the pair $(X,Y(\cdot))$ consisting of the random variable $X$ and the process $Y(\cdot)$) is adapted. We think of $\mathbb G$ as the ``initial enlargement" of $\mathbb F$
by the random variable $X$, and denote the ``ultimate" $\sigma$-algebras of the filtrations $\mathbb F$ and $\mathbb G$, respectively, by 
\[
\F(\infty):=\sigma \big(\cup_{0\leq t<\infty}\F(t)\big)\,,  \qquad 
\mathcal G(\infty):=\sigma \big(\cup_{0\leq t<\infty}\mathcal G(t)\big)\,.
\]

\smallskip
We consider also the collection $\,\U$ of all $\mathbb F-$progressively-measurable 
processes $u(\cdot)=(u(t))_{0\leq t<\infty}$ with values in $(0,1]$, that satisfy for {\it every} $\omega\in\Omega$ the nondegeneracy condition
\begin{equation}
\label{2.1}
\varlimsup_{t\to\infty}\,\frac{1}{\,t\,}\int_0^t u(s,\omega)\,\ud s> 0\,.
\end{equation}

  For each such ``control process" $u(\cdot)\in\U $ and each $t\in[0,\infty)$, we introduce the   measure $\,\P^u_t\sim\P\,$ on $(\Omega,\G(t))$ via
\begin{equation}
\label{2.2}
\left.\frac{\ud \P_t^u}{\ud \P}\right\vert_{\G(t)} = \Lambda^u (t):=\exp\left(X\int_0^t u(s)\, \ud Y(s)-\frac{X^2}{2}\int_0^t u^2(s)\, \ud s\right)\,, \quad 0 \le t < \infty\,.
\end{equation}
The resulting process $\Lambda^u(\cdot)=\big(\Lambda^u(t)\big)_{0\leq t<\infty}$ is clearly a 
$(\P/\mathbb G)$-local martingale and supermartingale. In fact, it is also a  {\it martingale,} as it has constant expectation 
\[
\E^\P\left[ \Lambda^u(t)\right] =
\int_\R \E^\P\left[\exp\left(b\int^t_0 u(s)\, \ud Y(s)-\frac{b^2}{2}\int^t_0 u^2(s)\, \ud s\right)\right]  {\bm \mu}   (\ud b) =1\,, \quad 0 \le t < \infty
\]
on account of the independence of $X$ and $Y(\cdot)$ under $\P \,$, and of the boundedness of $u(\cdot)$. Consequently,  each measure $\P^u_t$ as in \eqref{2.2} is in fact a probability measure on $(\Omega,\G(t))$. 

The theory of the so-called \textsc{F\"ollmer} measure (\cite{F1}; 
see also the Discussion on page 192 of \cite{KS1}) provides now the existence 
of a probability measure $\P^u$ on $\G(\infty)$, which agrees with $\P^u_t$ on 
$\G(t)$, for every $t\in[0,\infty)$. This theory needs certain topological assumptions on the 
underlying space $\Omega$ and  filtration $\mathbb G$, which are anything but onerous and are tacitly assumed here. We stress   that the measure $\P^u$ need  not be absolutely continuous with respect to $\P$ on $\,\G(\infty)$.

The \textsc{Girsanov} theorem (cf.\,\cite[Theorem 3.5.1]{KS1}) allows us now to write the dynamics 
\begin{equation}
\label{2.3}
Y(t)= X\int_0^t u(s)\, \ud s + W^u(t), \qquad 0\leq t<\infty,
\end{equation}
with $W^u(\cdot)=\big(W^u(t)\big)_{0\leq t<\infty}$ a standard, scalar $(\P^u/\mathbb G)-$Brownian motion on $(\Omega, \G(\infty))$. In particular, this Brownian motion is independent of 
$\,\G(0)=\sigma(X)$ under $\P^u$.

\smallskip
We interpret  the equation \eqref{2.3} as positing that {\it we observe the random variable $X$ in a bath of white noise;}  and that, based on the record  $\F(t)
$ of observations $\,Y(s), ~0 \le s \le t\,$ accumulated up to any given time $t\in[0,\infty)$, we can select the ``instantaneous intensity" $u(t)\in(0,1]$ of these observations.  We denote then by 
\begin{equation}
\label{2.4}
\widehat X^u(t) := \E^{\P^u}\big[X\,\big \vert \,\F(t)\big]=\E^{\P^u_t}\big[X\,\big \vert \, \F(t)\big]\end{equation}
the $\P^u-$conditional expectation (least-squares estimate) of $X$ at time $t \in [0, \infty),$  given the observations $\F(t)$ up to that time; and by
\begin{equation}
\label{2.5}
V^u (t):= \Var^{\P^u}\big(X\, \big \vert \, \F(t)\big) \,= \,\E^{\P^u}\left[\big(X-\widehat X^u(t) \big)^2 \,\big \vert \,\F(t)\right] \,= \,
\E^{\P^u_t}\left[\big(X-\widehat X^u(t) \big)^2 \,\big \vert \,\F(t)\right]
\end{equation}
the $\P^u-$conditional variance of $X$, given $\F(t)$.

\subsection{The Problem}
\label{sec2.1}

We wish to estimate the unobservable quantity $X$ ``as faithfully as possible", by trying to keep the conditional variance in \eqref{2.5} as small as we can. But suppose that observation is costly, and proportional to the square of the intensity $u(t)$ in \eqref{2.3}, per unit of time. Then, observing over a long time-interval, and using a large intensity $u(\cdot)$, increases the fidelity of estimation  but also the incurred cost. The question we study, then,  is  {\it how   to balance the  two competing objectives,} of fidelity in estimation and   cost minimization,  
{\it in an optimal fashion, when we can also choose the termination time of the experiment.}

\smallskip
For that, we need a criterion. With a given real constant $c>0$ measuring the weight we assign to the cost of observation per unit of time, we set out to minimize the expected total cost 
\begin{equation}
\label{2.6}
\E^{\P^u}\left[ \Big(X-\widehat X^u(\tau) \Big)^2 + c\int_0^\tau u^2(t)\, \ud t \, \right]
\end{equation}
associated with employing a control $u(\cdot)$ and a stop-rule $\tau,$ over all pairs $(\tau,u(\cdot))\in\mathcal T\times\U$. Here 
$\mathcal T$ is the collection of stopping times of the observations filtration $\mathbb F=\big(\F(t)\big)_{0\leq t<\infty}\,$, 
and $\,\U$ the collection of control processes introduced at the start of the present section.

\subsection{A Modified Criterion}
\label{sec2.2}

A bit more generally, we consider also the problem of minimizing, again over all pairs $(\tau,u(\cdot))\in\mathcal T\times\U$, the expected total cost 
\begin{equation}
\label{2.7}
\E^{\P^u}\left[ \Big(X-\widehat X^u(\tau)\Big)^2 + \int_0^\tau h \big(u(t)\big)\, \ud t \,\right].
\end{equation}
Here $h:(0,1]\to(0,\infty)$ is a continuous and non-decreasing function, for which we assume the existence of a number  $u_0\in(0,1]$ with the property   
\begin{equation}
\label{2.8.5}
\frac{h(u)}{u^2}\geq \frac{h(u_0)}{u_0^2}\,, \quad \forall ~ u\in(0,1].
\end{equation}
\begin{example} {\it (A Superquadratic Cost of Control):} 
When  \eqref{2.8.5} holds with $u_0=1$, and setting $c:=h(1)$,   the above condition becomes
\begin{equation}
\label{2.8}
h(u)\geq c\,u^2,\quad \forall ~u\in(0,1]\,.
\end{equation}
We express 
this special case by saying that it imposes a 
 {\it superquadratic cost of control}.

This dispensation includes, for instance, a linear cost $h(u)=cu$ of observation per unit of time; or more generally, a cost of the form $\, h (u) =\, c \,u^p\,$ for $0<p<2$. 
\end{example}

\section{Elementary Filtering Theory}
\label{sec3}

It is fairly clear, from \eqref{2.2} and the \textsc{Bayes} rule, that the conditional (``posterior") distribution of $X$, under the probability measure $\P^u$ and given $\F(t)$, is
\begin{equation}
\label{3.1}
\P^u \big(X\in B\,\big \vert \, \F(t)\big)= 
{\bm \mu}_{A^u(t),\,Z^u(t)}(B),\qquad B\in\mathcal B(\R)
\end{equation}
for every given $t \in [0, \infty)$. We are invoking in this expression the random variables 
\begin{equation}
\label{3.2}
A^u(t):=\int_0^tu^2(s)\ud s, \quad \quad Z^u(t):=\int_0^t u(s)\, \ud Y(s),
\end{equation}
and the family of probability measures 
\begin{equation}
\label{3.3}
{\bm \mu}_{\,\theta,\zeta}(B):=\frac{1}{F(\theta,\zeta)}\int_B \exp\left(b \,\zeta-\frac{b^2}{2}\,\theta\right){\bm \mu}(\ud b), \qquad (\theta,\zeta)\in(0,\infty)\times\R.
\end{equation}
The normalizer 
\begin{equation}
\label{3.4}
F(\theta,\zeta):=\int_\R \exp\left(b\,\zeta-\frac{b^2}{2}\,\theta\right){\bm \mu}(\ud b)
\end{equation}
in (\ref{3.3}), is the so-called \textsc{Widder} transform (\cite{W};   \S4.3.B in \cite{KS1}) of the probability measure ${\bm \mu}$ at $(\theta,\zeta)$;
and the resulting function solves on $(0,\infty)\times\R$  the backwards heat equation 
\begin{equation}
\label{3.5}
\partial F + \frac{1}{\,2\,} \,\D^2 F=0\,.
\end{equation}
Here and throughout this paper, we denote by $\partial$ and $\D$ differentiation with respect to 
the temporal and the spatial argument (here $\theta$ and $\zeta$), respectively.

The logarithmic gradient
\begin{equation}
\label{3.6}
G(\theta,\zeta):=\D\log F(\theta,\zeta)=\int_\R b\,{\bm \mu}_{\theta,\zeta}(\ud b)
\end{equation}
of the \textsc{Widder} transform in \eqref{3.4} gives the center of gravity of the distribution in \eqref{3.3}, and   solves on $(0,\infty)\times\R$ the backwards \textsc{Burgers} equation
\begin{equation}
\label{3.7}
\partial G + \frac{1}{\,2\,}\,\D^2 G + G\cdot \D G =0\,.
\end{equation}
Whereas, the gradient 
\begin{equation}
\label{3.8}
H(\theta,\zeta):=\D G(\theta,\zeta)=\int_\R \Big(b-G(\theta,\zeta)\Big)^2 {\bm \mu}_{\theta,\zeta}(\ud b)
=\int_\R b^2 {\bm \mu}_{\theta,\zeta}(\ud b) -G^2(\theta,\zeta)
\end{equation}
of this new function $G,$ is the second central moment of the distribution ${\bm \mu}_{\theta,\zeta}$ in \eqref{3.3}  and   solves  on $(0,\infty)\times\R$ the equation
\begin{equation}
\label{3.9}
\partial H + \frac{\,1\,}{2}\, \D^2 H + G\cdot \D H + H^2 =0\,.
\end{equation}

In terms of these functions, and of the random quantities in \eqref{3.2}, 
the posterior mean and variance of \eqref{2.4}/\eqref{2.5} are expressed concisely as
\begin{equation}
\label{3.10}
\widehat X^u(t)= \E^{\P^u} \big[X \, \big \vert \, \F(t)\big] = G \big(A^u(t), Z^u(t)\big)
\end{equation}
and 
\begin{equation}
\label{3.11}
 V^u (t)=\Var^{\P^u} \big(X \, \big \vert \, \F(t)\big) = \E^{\P^u}\left[ \Big(X-\widehat X^u(t)\Big)^2\, \big \vert \,\F(t)\right]= H \big(A^u(t), Z^u(t)\big).
\end{equation}
These two representations will be very useful in what follows.

\subsection{Innovations}

We introduce now, for each given control process $u(\cdot)\in\U$, the so-called {\it innovation process} 
\begin{equation}
\label{3.12}
N^u(t):= Y(t)-\int_0^t\widehat X^u(s) \,u(s)\,\ud s = W^u(t)+\int_0^t \big(X-\widehat X^u(s)\big)\,u(s)\,\ud  s  
\end{equation}
for $0\leq t<\infty.$ This is fairly easily checked to be a $(\P^u/\mathbb F)-$martingale. It has also continuous paths, 
and its quadratic variation over an interval $[0,t]$ is equal to the length $t$ of the interval, so this process is a
$(\P^u/\mathbb F)-$Brownian motion by a   result of P.\,\textsc{L\'evy} (Theorem 3.3.16 in \cite{KS1}).

In terms of this process $N^u(\cdot)$, and  in conjunction with \eqref{3.10}, \eqref{3.7} and some fairly elementary stochastic calculus, the $(\P^u/\mathbb F)-$martingale 
$$
\widehat X^u(t)= \E^{\P^u} \big[X \, \big \vert \, \F(t) \big] =\int_\R b\,{\bm \mu}_{A^u(t),Z^u(t)}(\ud b)\,, \qquad 0\leq t <\infty
$$ 
of \eqref{2.4}, \eqref{3.10} is seen to have the dynamics 
\begin{equation}
\label{3.13}
\ud \widehat X^u(t)= H\big(A^u(t),Z^u(t) \big)\,u(t) \, \ud N^u(t).
\end{equation}

\subsection{Re-Parametrization}

Let us denote by $\I$ the interior of the smallest closed interval that contains the support of the probability measure ${\bm \mu}$. Because ${\bm \mu}$ has strictly positive variance, the probability measure ${\bm \mu}_{(\theta, \zeta)}$ in \eqref{3.3} is not a \textsc{Dirac}   mass $\,{\bm \delta}_{G(\theta, \zeta)}\,$,    so the quantity $ H(\theta,\zeta)=\D G(\theta,\zeta)$ in \eqref{3.8} is strictly positive. As a consequence, for every fixed $\theta\in(0,\infty)$ the continuous function
$$\R\ni \zeta \longmapsto G_\theta(z):=G(\theta,\zeta)\in \I$$ in \eqref{3.6} is strictly increasing. Denoting the inverse of this function by $G_\theta^{-1}(\cdot)$, we re-cast \eqref{3.10}
as 
\begin{equation}
\label{3.14}
Z^u(t)=G^{\,-1}_{A^u(t)} \big(\widehat X^u(t)\big), \qquad 0\leq t<\infty.
\end{equation}

We introduce also the composite function 
\begin{equation}
\label{3.15}
(0,\infty)\times\I \,\ni(\theta,x)\longmapsto\Psi(\theta,x):=H \big(\theta,G_\theta^{-1}(x)\big)\in(0,\infty).
\end{equation}
This   solves on the strip $\,(0,\infty)\times\I\,$ the (fully nonlinear)  equation of parabolic type 
\begin{equation}
\label{3.16}
\partial \Psi + \Psi^2 \Big(1+\frac{1}{\,2\,}\,\D^2\Psi \Big)=0\,,
\end{equation}
  and its temporal derivative is non-positive there (cf.\,\cite[Proposition 3.6]{EV}):
\begin{equation}
\label{3.17}
\partial \Psi\leq 0, \quad\mbox{ equivalently }\quad\D^2\Psi\geq -2\,, \qquad \text{on}~ (0, \infty) \times {\cal I}\,.
\end{equation}
It is worth noting that this equation \eqref{3.16} for $\Psi$, just like the equations \eqref{3.5} for $F$ and \eqref{3.7} for $G$,  is {\it autonomous;} and that this is {\it not} the case for the equation  \eqref{3.9} that governs the function $H$, which needs the function $G$ as its input.

In the light of \eqref{3.14} and \eqref{3.15}, we can cast the conditional  (posterior) variance process of \eqref{2.5}/\eqref{3.11} in terms of the function $\Psi$ in \eqref{3.16}, as 
\begin{equation}
\label{3.18}
V^u(t)=\Var^{\P^u} \big(X\,\big \vert \, \F(t)\big) = H \big(A^u(t),Z^u(t)\big)= \Psi \big(A^u(t),\widehat X^u(t)\big),\qquad 0\leq t <\infty
\end{equation}
and the dynamics of \eqref{3.13} for the $(\P^u/\mathbb F)-$martingale  $\widehat X^u(\cdot)$ as
\begin{equation}
\label{3.19}
\ud \widehat X^u(t) = \Psi \big(A^u(t),\widehat X^u(t)\big) \,u(t) \,\ud N^u(t).
\end{equation}

\subsection{Time-Change}

It makes 
good sense at this point, to look at said dynamics under the lens of a time-change, as follows. We introduce the square-integrable $(\P^u/\mathbb F)-$martingale 
\begin{equation}
\label{3.20}
M^u(t):=\int_0^tu(r) \, \ud N^u(r),\quad 0\leq t<\infty,
\end{equation}
and note its quadratic variation 
\begin{equation}
\label{3.20a}
\langle M^u\rangle (t) = \int_0^t u^2(r)\,\ud r = A^u(t)\leq t\,.
\end{equation}

On the strength of the \textsc{Dambis-Dubins-Schwarz} theorem,  this martingale can be cast as 
\[M^u(t)=B^u \big(A^u(t)\big), \quad 0\leq t<\infty\]
in terms of a suitable $\P^u-$Brownian motion $B^u(\cdot)$.
In fact, from the non-degeneracy condition \eqref{2.1} and the \textsc{Cauchy-Schwarz} inequality, we obtain 
\[
0<\varlimsup_{t\to\infty}\left(\frac{1}{t}\int^t_0 u(r,\omega)\ud r\right)^2
\leq \, \varlimsup_{t\to\infty}\left(\frac{1}{t}\,A^u(t,\omega)\right).
\]
It  follows   that the condition \eqref{2.1} implies $A^u(\infty,\omega)=\infty$ for every 
$\omega\in\Omega\,$; thus, we can invoke Theorem 3.4.6 in \cite{KS1}, and express  the $\P^u-$Brownian motion $B^u (\cdot)$ above as 
\begin{equation}
\label{3.21}
B^u(s)= M^u \big(T^u(s)\big), \qquad T^u(s):=\inf\{t\geq 0:A^u(t)\geq s\}
\end{equation}
for $0\leq s<\infty$\,. 

Likewise, we introduce the time-changed process
\begin{equation}
\label{3.22}
Q^u(s):=\widehat X^u \big(T^u(s)\big), \qquad 0\leq s<\infty 
\end{equation}
and note that, in terms of it, the dynamics of \eqref{3.19} are recast as the diffusion equation
\begin{equation}
\label{3.23}
\ud Q^u(s) = \Psi \big(s,Q^u(s)\big)\,\ud B^u(s),\qquad Q^u(0)= m:=\int_\R b\,{\bm \mu}(\ud b)\,.
\end{equation}

It is important here to note that the stochastic equation \eqref{3.23}, driven by the $\P^u-$Brownian motion $B^u (\cdot)$, admits a strong 
solution which is unique pathwise, thus also in distribution; in particular, {\it the distribution of the diffusion process $Q^u(\cdot)$    does not depend on the control} 
$\,u(\cdot)\in \U$.

\subsection{Filtrations and   Stopping Times}

The following result elucidates the structure of the various filtrations. Although we do not use it directly below, we find it to be of considerable intrinsic 
interest: it states that the reformulation/reparametrization carried out in section \ref{sec4} does not result in diminution or enlargement of the observations filtration $\mathbb{F}$.

\smallskip
Here and below, we denote by $\, \mathbb F^{Z}=  (\F^{Z}(t) )_{0\leq t<\infty  }\,$ the   filtration generated by a given process  $Z=   (  Z(t) )_{0 \le t < \infty} \,$  with values in some Euclidean space: namely,   the smallest right-continuous filtration  to which the given process $Z$ is adapted. (With this notation, we have in fact $\,\mathbb F  \equiv \mathbb F^{Y}$ and $\,\mathbb G  \equiv \mathbb F^{(X,Y)}$  in the context of  section \ref{sec2}.)  

\begin{proposition} {\bf Sufficient Statistic.} 
\label{prop3.1}
For every control process $u(\cdot)\in\U$, the pair 
of processes $\big(A^u(\cdot),\widehat X^u(\cdot)\big)$  from \eqref{2.4}/\eqref{3.2} constitutes a sufficient statistic for the observation filtration $\mathbb F$, in the sense that it generates the same filtration as the observations process:
\begin{equation}
\label{3.24}
\F^{\,(A^u, \,\widehat X^u)}(t)\,=\,\F^Y(t)\,\equiv \,\F(t), \qquad 0\leq t<\infty.
\end{equation}
\end{proposition}

\begin{proof}
Both random variables $A^u(t)$, $\widehat X^u(t)$ from \eqref{3.2}, \eqref{2.4} are 
$\F(t)-$measurable, for each $t\in[0,\infty)$, so the inclusion 
 $\,\F^{\,(A^u, \,\widehat X^u)}(t)\,\subseteq\,\F^Y(t)\,   $  in 
\eqref{3.24} is evident. On the other hand, \eqref{3.2} and \eqref{3.14} give
\[Y(t)= \int_0^t\frac{\ud Z^u(r)}{u(r)}
=\int_0^t \left(\frac{\ud}{\ud r} A^u(r)\right)^{-1/2} \ud G^{\,-1}_{A^u(r)}\big(\widehat X^u(r)\big).\]
This shows that, for every $t\in[0,\infty)$, the random variable $\,Y(t)\,$ is $\,\F^{\,(A^u,\,\widehat X^u)}(t)-$measurable, so the   inclusion  $\,\F^{\,(A^u, \, \widehat X^u)}(t)\,\supseteq\,\F^Y(t)\,   $ in \eqref{3.24} follows as well.
\end{proof}

Down the road, we shall need the following result regarding stopping times of different filtrations. We recall for this purpose the $\mathbb F-$adapted, continuous and strictly increasing process $A^u(\cdot)$ from \eqref{3.2}, its inverse $T^u(\cdot)$ from \eqref{3.21}, the diffusion process $Q^u(\cdot)$ of \eqref{3.22}/\eqref{3.23}, and the collection $  \mathcal T $ of stopping times of the observation filtration $\mathbb F$. 

\begin{proposition}
\label{prop3.2}
(i) If a random time $\tau:\Omega\to [0,\infty]$ is a stopping time of the observations filtration
$\mathbb F$, then, for any given control $u(\cdot)\in\U$, the random time $A^u(\tau)$ is 
a stopping time of the filtration 
\begin{equation}
\label{3.25}
\mathcal H^u(s):=\F\big(T^u(s)\big), \qquad 0\leq s<\infty\,.
\end{equation}
This is larger than the filtration
$\, \mathbb F^{Q^u}= \big(\F^{Q^u}(s)\big)_{0\leq s<\infty  } \,$   
 generated by the diffusion process $Q^u(\cdot)$ in \eqref{3.22}/\eqref{3.23}.  
 
\smallskip
\noindent
 (ii) Conversely, for any given control $u(\cdot)\in\U$ and every  stopping time $\, \rho \,$ of the filtration $\mathbb H^u= \big(\mathcal H^u(s)\big)_{0\leq s<\infty}$ in \eqref{3.25}, we have     $\,\tau:=T^u(\rho)\in\mathcal T$.
\end{proposition}

\begin{proof}
(i) For any given control $u(\cdot)\in\U$, the strictly increasing change-of-clock process $A^u(\cdot)$
in \eqref{3.2} is clearly adapted to the observations filtration $\mathbb F$, so each 
$T^u(s)$ in \eqref{3.21} is a stopping time of this filtration. Thus, from basic properties of 
filtrations and stopping times (cf.\,\cite{KS1}, Lemma 1.2.16), 
\begin{equation}
\label{3.26a}
\{A^u(\tau)\leq s\} = \{\tau\leq T^u(s)\}\in\F\big(T^u(s)\big) = \mathcal H^u(s)
\end{equation}
holds for every $s\in[0,\infty)$, so $A^u(\tau)$ is a stopping time of the filtration $\,\mathbb H^u= \big(\mathcal H^u(s)\big)_{0\leq s<\infty}\,$ in 
\eqref{3.25}. The process $Q^u(\cdot)$ of \eqref{3.22} is clearly adapted to this filtration $\,\mathbb H^u  ,$  
because the process $\widehat X^u(\cdot)$ in \eqref{2.4} is adapted to the filtration $\mathbb F$.

\smallskip
\noindent
(ii) Similarly, the string 
\[
\{\tau\leq t\}=\{T^u(\rho)\leq t\} = \{\rho\leq A^u(t)\}\in \mathcal H^u\big( A^u(t) \big)=   \F \big(t\big)\,, \qquad 0\leq t<\infty,\]
  shows that $\tau\in\mathcal T$.
\end{proof}

\section{Reformulation}
\label{sec4}

The considerations of the previous section allow us to carry out a detailed study of the posterior  variance process 
\[
V^u(t)= \Var^{\P^u} \big(X \, \big \vert \, \F(t)\big)= \E^{\P^u_t}\left[ \Big(X-\widehat X^u(t)\Big)^2 \, \Big \vert \, \F(t)\right], \qquad 0\leq t<\infty
\]
in \eqref{2.5}/\eqref{3.11}. This then leads to a reformulation of the problems in 
subsections~\ref{sec2.1}/\ref{sec2.2}, that will make them amenable to analysis, and eventually even to computation.

We start by noting that the dynamics of \eqref{3.19} imply, in conjunction with \eqref{3.18}/\eqref{3.16},   the dynamics 
\[
dV^u(t)= -\Psi^2 \big(A^u(t), \widehat X^u(t)\big)\,\ud A^u(t) + (\Psi\cdot \D\Psi) \big(A^u(t),\widehat X^u(t)\big)\,\ud M^u(t)
\]
for the conditional variance process in \eqref{2.5}/\eqref{3.11}. It follows that the positive process
\begin{equation}
\label{4.2}
V^u(t)+\int_0^t \big( V^u(s)\big)^2\ud A^u(s) = 
V(0)+ \int_0^t(\Psi\cdot \D\Psi) \big(A^u(s),\widehat X^u(s)\big)\,\ud M^u(s), \quad 0\leq t<\infty
\end{equation}
where $\,V(0) = \Var (X),$ is a $(\P^u/\mathbb F)-$local martingale, and therefore a supermartingale; thus 
\begin{equation}
\label{4.3}
\E^{\P^u}\left[ V^u(\tau) + \int_0^\tau \big( V^u(t)\big)^2 \, \ud A^u(t)\right]\leq 
V(0)<\infty
\end{equation}
holds for every stopping time $\tau\in\mathcal T$, in particular, 
\begin{equation}
\label{4.4}
\E^{\P^u}\left[ \int_0^\infty \big( V^u(t)\big)^2\,\ud A^u(t)\right]\leq 
V(0)<\infty.
\end{equation}

\begin{proposition}
\label{prop4.1}
The first inequalities in each of \eqref{4.3}/\eqref{4.4} hold as equalities, and the process of 
\eqref{4.2} is a true $(\P^u/\mathbb F)-$martingale.
\end{proposition}

\begin{proof}
The strong law of large numbers for the $(\P^u/\mathbb G)-$Brownian motion $W^u(\cdot)$ of  \eqref{2.3} shows,
in conjunction with the property \eqref{2.1}, that 
\[X=\frac{\varlimsup_{t\to\infty} (Y(t)/t)}{\,\,\varlimsup_{t\to\infty} \frac{1}{\,t\,}\int_0^t u(s)\,\ud s\,\,} \quad \mbox{ holds }\,\P^u-\mbox{a.e}.\]
Thus, $X$ is measurable with respect to the $\P^u-$completion $\F^u(\infty)$ of the $\sigma$-algebra $\F(\infty)$,
and the \textsc{P.\,L\'evy} martingale convergence theorem (cf.\,Theorem 9.4.8 in \cite{Ch}) gives
\[\lim_{t\to\infty}\E^{\P^u}\left[X^k \, \big   \vert  \, \F(t)\right]=\E^{\P^u}\left[X^k \, \big \vert \, \F^u (\infty)\right]= X^k,\quad k=1,2.\]
Consequently, the conditional variance 
\[
V^u(t)\,=\, H \big(A^u(t), Z^u(t)\big)\,=\,\Var^{\P^u} \big(X \, \big  \vert \, \F(t)\big)\,=\, \E^{\P^u}\left[X^2 \, \big \vert \, \F(t)\right]-\left(\E^{\P^u}\big[X \, \big \vert \, \F(t)\big]\right)^2
\]
from \eqref{3.11} converges $\P^u-$a.e.\,to zero, as $t\to\infty$; the $\P^u-$martingale $\widehat X^u(\cdot)$ in \eqref{2.4}, as well as the 
$\P^u-$submartingale $\big(\widehat X^u(\cdot)\big)^2$, are both uniformly integrable; and the  representation 
\begin{equation}
\label{4.5}
\widehat X^u(\tau)\,= \,m +\int_0^\tau\Psi \big(A^u(t),\widehat X^u(t)\big)\,u(t) \,\ud N^u(t)\,,
\end{equation}
with $m=\E[X]=\int_\R \,b \, {\bm \mu} (\ud b)$ as in \eqref{3.23}, holds $\P^u-$a.e., for every stopping time $\tau\in\mathcal T$, {\it including} $\tau=\infty$.
We deduce 
\[X-\widehat X^u(\tau)=\int_\tau^\infty\Psi \big(A^u(t),\widehat X^u(t)\big)\, u(t) \,\ud N^u(t),\]
thus also 
\begin{eqnarray*}
V^u(\tau) &=& \Var^{\P^u} \big(X\, \big \vert \, \F(\tau)\big)=\E^{\P^u}\left[ \big(X-\widehat X^u(\tau)\big)^2\, \big \vert \, \F(\tau)\right]\\
&=& \E^{\P^u}\left[\left.\int_\tau^\infty \Psi^2 \big(A^u(t),\widehat X^u(t)\big)\,\ud A^u(t)\,\right\vert\F(\tau)\right] = \Psi \big(A^u(\tau),\widehat X^u(\tau)\big);
\end{eqnarray*}
whereas, taking $\P^u$-expectations, the equalities 
\[\E^{\P^u}\left[\Var^{\P^u} \big(X \, \big \vert \, \F(\tau)\big)\right] =\E^{\P^u}\left[ \int_\tau^\infty \Psi^2 \big(A^u(t),\widehat X^u(t)\big) \, \ud A^u(t)\right] 
=\E^{\P^u}\left[\Psi \big(A^u(\tau),\widehat X^u(\tau)\big) \right]\]
are seen to hold  as well. We have used here the finite upper bound in \eqref{4.3}-\eqref{4.4}, which, in conjunction with \eqref{4.5}, yields also 
\[
\Var^{\P^u}\left(\E^{\P^u}\big[X\, \big \vert \,\F(\tau)\,  \big]\,\right) = \Var^{\P^u}\left(\widehat X^{u}(\tau)\right)
=\E^{\P^u}\left[\int_0^\tau\Psi^2 \big(A^u(t),\widehat X^u(t)\big)\, \ud A^u(t)\right].
\]

We recall at this point a classical identity about the variance of a square-integrable random variable (to the effect that it is equal to the sum, of 
the expectation of the conditional variance, plus the variance of the conditional expectation), and obtain
\begin{eqnarray*}
\Var(X) &=& \E^{\P^u}\left[\Var^{\P^u} \big(X\,\big \vert \,\F(\tau)\big)\right] + \Var^{\P^u}\left(\E^{\P^u}\left[X\, \vert \, \F(\tau)\right]\right)\\
&=& \E^{\P^u}\left[\Psi \big(A^u(\tau),\widehat X^u(\tau)\big) + \int_0^\tau\Psi^2 \big(A^u(t),\widehat X^u(t)\big)\ud A^u(t)\right],
\end{eqnarray*}
as well as 
\[
V(0)\,=\,\Var (X)\, = \,\E^{\P^u}\left[\int_0^\infty \Psi^2 \big(A^u(t),\widehat X^u(t)\big)\,\ud A^u(t)\right]
\]
upon taking $\tau=\infty$. Consequently, the first inequality in each of \eqref{4.3}, \eqref{4.4} holds as equality.
\end{proof}

\subsection{Consequences}
\label{sec4.0}

It is now an immediate consequence of Proposition~\ref{prop4.1}, that the expected cost in \eqref{2.6} can be written as
\begin{eqnarray}
\label{4.5.5}
\E^{\P^u}\left[ \big(X-\widehat X^u(\tau)\big)^2 + c\int_0^\tau u^2(t)\,\ud t\right]
&=& \E^{\P^u}\big[V^u(\tau)+ cA^u(\tau)\big] \\
\notag
&=&
V(0)+ \E^{\P^u}\left[\int_0^\tau\Big(c-\big( V^u(t)\big)^2\Big) \ud A^u(t)\right].
\end{eqnarray}
$\bullet~$ Thus, on account of \eqref{3.18}, \eqref{3.2} and \eqref{3.21}--\eqref{3.23}, the problem of subsection~\ref{sec2.1} can be cast equivalently as minimizing, over all pairs 
$(\tau, u(\cdot))\in \mathcal T\times \U$, the expectation 
\begin{equation}
\label{4.6}
\E^{\P^u}\left[\int_0^\tau\left(c-\Psi^2 \big(A^u(t),\widehat X^u(t)\big)\right)\ud A^u(t)\right]
\,=\,\E^{\P^u}\left[\int_0^{A^u(\tau)}\Big(c-\Psi^2 \big(s,Q^u(s)\big)\Big)\,\ud s\,\right].
\end{equation}
$\bullet~$  Likewise, the more general problem of subsection~\ref{sec2.2} amounts to minimizing, over all pairs $(\tau, u(\cdot))\in \mathcal T\times \U$, the expectation 
$$
\E^{\P^u}\left[\int_0^\tau\left(\frac{h(u(t))}{u^2(t)}-\Psi^2 \big(A^u(t),\widehat X^u(t) \big)\right) \ud A^u(t)\right] =~~~~~~~~~~~~~
$$
 \begin{equation}
\label{4.7}
~~~~~~~~~~~~~~~~=\,\E^{\P^u}\left[\int_0^{A^u(\tau)}\left(\frac{h(u(T^u(s)))}{u^2(T^u(s))}-\Psi^2 \big(s,Q^u(s)\big)\right) \ud s\right].
\end{equation}

\section{Results}
\label{sec5}

Let us denote now by $\widehat X(\cdot):= \widehat X^1(\cdot)=Q^1(\cdot)$ the 
$\P^*-$martingale of \eqref{2.4} and \eqref{3.22}/\eqref{3.23}, corresponding to the 
{\it ``full-bang" control }$$u^*(\cdot)\equiv 1\,, $$ with  the identification $\P^*\equiv \P^1,$    dynamics
\begin{equation}
\label{4.8a}
\ud \widehat X (t ) = \Psi \big(t ,\widehat X (t ) \big)\,\ud \widehat N(t)\,, \qquad \widehat X (0 )=m
\end{equation}
in the manner of \eqref{4.5}, \eqref{3.23} 
 for a diffusion in natural scale and values in ${\cal I},$  and $\widehat N (\cdot )$  a $\P^*-$Brownian motion. 
 
 We denote also by $\tau^*$ the smallest stopping time which minimizes 
the expected cost
\[\E^{\P^*}\left[\Psi \big(\tau,\widehat X(\tau)\big)+c\tau\right],\]
or equivalently the expectation
\begin{equation}
\label{4.8}
\E^{\P^*}\left[ \, \int_0^\tau\left(c -\Psi^2 \big(t,\widehat X(t)\big)\right)\ud t \,\right],
\end{equation}
over all stopping times $\tau\in\mathcal T$ and always with the identification $\P^*\equiv \P^1$.  

Such a stopping time turns out to exist, and indeed to have the form (\ref{4.8_st}) below.   
It is clear also from \eqref{4.5.5}, that the infimum over $\tau\in\mathcal T$ of the quantity 
in \eqref{4.8} takes values in $[-\Var(X),0]$.

\subsection{The Problem of Minimizing \eqref{4.8} Subject to  \eqref{4.8a}, in  \cite{EKV}}
\label{sec4.1_0}

The theory of optimal stopping for \textsc{Markov} processes is a well-developed subject, accessible in several sources, for instance in \textsc{Peskir \& Shiryaev}  \cite{PS}. 
The particular problem of minimizing the expectation in \eqref{4.8}, over stopping times of the filtration generated by the one-dimensional diffusion \eqref{4.8a} in natural scale, is studied in detail in \cite{EKV}.

\smallskip
Let us elaborate. Using the Markovian nature of the process $\widehat X(\cdot)$, 
we cast the problem of minimizing the expected cost in \eqref{4.8} in terms of the function 
\[
[0,\infty)\times \mathcal I\ni (s,x)\longmapsto v(s,x)=\inf_{\tau\in\mathcal T}\,
\E^{\P^*}\left[\int_0^\tau\left(c-\Psi^2 \big(t+s,\widehat X^{(s,x)}(t+s) \big)\right) \ud t\right].
\]
Here the minimization is subject to the dynamics 
\[
\ud \widehat X^{(s,x)}(t+s) = \Psi \big(t+s,\widehat X^{(s,x)}(t+s) \big)\,\ud \widehat N(t), \qquad t>0
\]
driven by the ``innovations process", the $\,(\P^*/\mathbb F)-$Brownian motion $\widehat N(\cdot)$, and subject to the initial condition $\widehat X^{(s,x)}(s)=x\in\mathcal I$, in the manner of \eqref{4.8a}.

\medskip
Then the process $\widehat X^{(0,m)}(\cdot)$, with $m=\int_\R b \, {\bm \mu} (\ud b)$, is the same as the process of \eqref{4.5} with full-bang control $u(\cdot)\equiv 1$; 
the optimal stopping region is 
\begin{equation}
\label{4.8a_st}
\D:= \big\{(s,x)\in[0,\infty)\times\mathcal I: \,v(s,x)=0 \big\};
 \end{equation}
and the time 
\begin{equation}
\label{4.8_st}
 \tau^*:=\inf \big\{t\geq 0\,:\, \big(t,\widehat X^{(0,m)}(t)\big)\in\D \big\} \in \mathcal T,
 \end{equation}
of  first entry   into this region, minimizes the expression of \eqref{4.8}\,---\,not only over the collection 
$\mathcal T=\mathcal  S(\mathbb F)$ of stopping times of the filtration $\mathbb F$, but also over
the collection $\mathcal  S(\mathbb G)$ of stopping times of  {\it any} filtration $\mathbb G,$ 
such that the collection  $\big( \Omega, {\cal A} , \mathbb{P}^*), $ $\mathbb G = \big( {\cal G} (t) \big)_{0 \le t < \infty}\,,$ $\big( \widehat X(\cdot), \widehat N(\cdot)\big)$ constitutes a weak solution of the stochastic equation (\ref{4.8a}); cf.\,Definition 5.3.1 in \cite{KS1}. 

Moreover, $\tau^*$ in  \eqref{4.8_st} is the {\it smallest} such stopping time. 
  
Rare examples of prior distributions ${\bm \mu}$, for which this optimal stopping region in \eqref{4.8a_st} can be found explicitly, are
provided in the earlier work \cite{EKV}. Extensions of these examples to the current setting with control, are 
studied in sections \ref{gauss} and \ref{Bernoulli} below.
It has been a major challenge for us to find additional such examples, and we leave this issue to future research.



\subsection{The Problem of Subsection~\ref{sec2.1}}
\label{sec4.1a}

We are ready to state and prove our first result. 

\begin{theorem}
\label{thm4.2}
The pair $(\tau^*,u^*(\cdot))\in\mathcal T\times\U$ as above, namely, $u^*(\cdot)\equiv 1$ 
and $\tau^*$ 
as in \eqref{4.8_st}
(which attains the infimum in \eqref{4.8} over $\tau\in\mathcal T$)
, is optimal for the 
problem of subsection~\ref{sec2.1}.
\end{theorem}

\begin{proof} 
As above, we denote by $\mathcal S(\mathbb G)$ the collection of stopping times of a 
generic filtration $\mathbb G$; observe that $\mathcal T=\mathcal S(\mathbb F)$; recall the filtration $\,\mathbb H^u= \big(\mathcal H^u(s)\big)_{0\leq s<\infty}\,$ from  
\eqref{3.25}; and note that for every $(\tau,u(\cdot))\in\mathcal T\times\U$  we have 
\begin{eqnarray}
\label{4.9}
\E^{\P^u}\left[\int_0^\tau\left(c-\Psi^2 \big(A^u(t),\widehat X^u(t) \big)\right)\,\ud A^u(t)\right] &=&
\E^{\P^u}\left[\int_0^{A^u(\tau)} \Big( c-\Psi^2 \big(s,Q^u(s)\big)\Big)\ud s\right] ~\\
\notag 
&\geq& 
\inf_{\rho\in \mathcal S(\mathbb H^u)} \E^{\P^u}\left[\int_0^{\rho}\Big( c-\Psi^2 \big(s,Q^u(s)\big)\Big)  \ud s\right]\\
&=& \notag
\inf_{\rho\in \mathcal S(\mathbb F^{Q^u})} \E^{\P^u}\left[\int_0^{\rho} \Big( c-\Psi^2 \big(s,Q^u(s)\big)\Big) \ud s\right]\\
&=& \notag
\inf_{\rho\in \mathcal S(\mathbb F^{Q^1})} \E^{\P^1}\left[\int_0^{\rho} \Big( c-\Psi^2 \big(s,Q^1 (s)\big)\Big)
\ud s\right]\\
&=& \notag
\inf_{\tau\in \mathcal S(\mathbb F^{\widehat X})} \E^{\P^1}\left[\int_0^{\tau}\left( c-\Psi^2 \big(t,\widehat X(t)\big)\right)\ud t\right]\\
&=& \notag
 \E^{\P^*}\left[\int_0^{\tau^*}\left( c-\Psi^2 \big(t,\widehat X(t) \big)\right)\ud t\right],
\end{eqnarray}
always with the identification $\mathbb P^*\equiv \mathbb P^1$.
 
Here, the inequality is a consequence of Proposition~\ref{prop3.2}. The second equality follows from the fact that stopping times of the filtration $\mathcal S(\mathbb F^{Q^u})$ are ``sufficient"
for minimizing the expected cost 
\begin{equation}
\label{ospa}
\E^{\P^u}\left[ \, \int_0^\rho
\Big( c-\Psi^2 \big(s,Q^u(s)\big)\Big)\,
\ud s \, \right]
\end{equation}
over the stopping times $\rho$ of any filtration $\mathbb G^u$, such as $\mathbb H^u= \big(\mathcal H^u(s)\big)_{0\leq s<\infty}$ in \eqref{3.25}, with the property that the collection  $\big( \Omega, {\cal A} , \mathbb{P}^u), $ $\mathbb G^u = \big( {\cal G}^u (t) \big)_{0 \le t < \infty}\,,$ $\big(Q^u (\cdot), B^u (\cdot)\big)\,$  constitutes a weak solution of the stochastic differential equation (\ref{3.23}).


The third equality is a consequence of the fact, noted in the paragraph right before Proposition~\ref{prop3.1}, that the $\P^u-$distribution of the diffusion process $Q^u(\cdot)$ is   the same for all control processes $u(\cdot)\in\U\,$; whereas the fourth and fifth equalities are evident. 

The claimed optimality of the pair $(\tau^*,u^*(\cdot))$ is now clear.
\end{proof}


\subsection{The Problem of Subsection~\ref{sec2.2}}
\label{sec4.1b}

Similar reasoning applies to the problem of subsection~\ref{sec2.2} which, as we noted, amounts to minimizing the expectation in \eqref{4.7}. 

Indeed, for any pair $(\tau,u(\cdot))\in\mathcal T\times\U$, and denoting $c:=h(u_0)/u_0^2\,,$ we have 
\begin{eqnarray*}
\E^{\P^u}\left[\int_0^\tau\left(\frac{h(u(t))}{u^2(t)}-\Psi^2(A^u(t),\widehat X^u(t))\right) \ud A^u(t)\right]
&\geq&
\E^{\P^u}\left[\int_0^{A^u(\tau)} \Big( c-\Psi^2 \big(s,Q^u(s)\big)\Big)\ud s\right]\\
&&\hspace{-40mm}\geq
\inf_{\tau\in \mathcal S(\mathbb F^{\widehat X^{u^{o}}})} 
\E^{\P^{u^{o}}}\left[\int_0^{\tau}\left( c-\Psi^2 \big(A^{u^{o}}(t),\widehat X^{u^{o}}(t)\big)\right)\ud A^{u^{o}}(t)\right]\\
&& \hspace{-40mm} \notag
=
\,\, \E^{\P^{u^{o}}}\left[\int_0^{\tau^{o}}\left( c-\Psi^2 \big(A^{u^{o}}(t),\widehat X^{u^{o}}(t) \big) \right)\ud A^{u^{o}}(t)\right].
\end{eqnarray*}
In the last two expressions of this display,  we   deploy  the constant control  $u^{o}(\cdot)\equiv u_0$ in $\,\U,$ and note $A^{u^{o}}(t) = u_0^2\, t,$ $\widehat X^{u^{o}}(t) = Q^{u^{o}}(u_0^2\, t) $. 

The first inequality in the above display   is a consequence of the assumption \eqref{2.8};  and  the second follows directly from the string \eqref{4.9}, with the constant control  $u^{o}(\cdot)\equiv u_0$ in $\,\U $ replacing the   control  $u^{*}(\cdot)\equiv 1$. 
In the last expression,  $\tau^{o}$ is  the smallest  optimal stopping time for the problem of minimizing the expected cost 
\begin{equation}
\label{osp}
\E^{\P^{u^{o}}}\left[\int_0^{\tau}\left( c-\Psi^2 \big(A^{u^{o}}(t),\widehat X^{u^{o}}(t) \big) \right)\ud A^{u^{o}}(t)\right]
\end{equation}
corresponding to the constant control  $u^{o}(\cdot)\equiv u_0$ in $\,\U $. 

\medskip
These considerations lead to the following result.

\begin{theorem}
\label{thm4.3}
The pair $(\tau^o,u^{o}(\cdot))\in\mathcal T\times \U$, where $u^{o}(\cdot)\equiv u_0$
and $\tau^o$ attains the infimum in \eqref{osp} over $\tau\in\mathcal T$, is optimal for the 
problem of subsection~\ref{sec2.2}. 
\end{theorem}

\begin{remark}
\label{rem4.4}
{\rm 
The optimal stopping problem in \eqref{osp} corresponds to the sequential least-squares estimation of $X$ from observations of a process $Y(t)= u_0Xt + W^0(t),~0 \le t < \infty$ in the manner of \eqref{2.3}, with $W^0 (\cdot)$ a Brownian motion and cost of observation $c$ 
per unit of time. By standard scaling properties, setting $\,W(t)=u_0\,W^0(t/u_0^2),~0 \le t < \infty\,$ yields also a standard Brownian motion, in terms of which 
we have the expression $Y(t)=\frac{1}{u_0}\left(X u_0^2 \,t + W(u_0^2\,t)\right)$.

\smallskip
Least-squares estimation of $X$ using observations of the process $u_0X s + W^0(s),~0 \le s \le t \, $ (up to time $t$), is thus equivalent to least-squares estimation using observations of 
$X \theta + W(\theta),~0 \le \theta \le u_0^2 t  $ (up to time $u_0^2 t)$. Thus, the problem of subsection~\ref{sec2.2} reduces to the problem studied in \cite{EKV}, 
but with cost of observation $cu_0^{-2}$ per unit of time. 
}
\end{remark}

\subsection{Stopping Fast, when the Cost of Control is  Superquadratic}
\label{sec4.2}

Suppose now that, in the context and proof of Theorem~\ref{thm4.2}, as well as of Theorem \ref{thm4.3}   with superquadratic cost of control, we single out  and fix  an arbitrary
control process $\widehat u(\cdot)\in\U$, rather than  $u^*(\cdot)\equiv 1$. 

\smallskip
We can replace then the fifth expression in the string \eqref{4.9} by 
\begin{equation}
\label{4.10}
\inf_{\rho\in \mathcal S(\mathbb F^{Q^{\widehat u}})} \E^{\P^{ \widehat u}}\left[\int_0^{\rho}\Big( c-\Psi^2 \big(s,Q^{\widehat u}(s)\big)\Big) \, \ud s\right] \,=\, \E^{\P^{ \widehat u}}\left[\int_0^{\widehat \rho}\Big( c-\Psi^2 \big(s,Q^{\widehat u}(s)\big)\Big) \, \ud s\right] .
\end{equation}
Here $\widehat \rho\,,$ the smallest optimal stopping time for the problem of maximizing \eqref{ospa} in the context of the $\P^{\,\widehat u}-$diffusion $Q^{\widehat u}(\cdot)$, has under $\P^{\, \widehat u}$ the same distribution
as $\tau^*$ of \eqref{4.8_st} has under $\P^*\equiv\P^1$:
\begin{equation}
\label{4.11}
\P^*(\tau^*>t) \, = \, \P^{\, \widehat u}\big(\widehat\rho>t\big) \, \leq \, \P^{\, \widehat u} \big(\widehat \tau>t\big), \qquad
0\leq t<\infty
\end{equation}
with $\,\widehat\tau:=T^{\widehat u}(\widehat\rho\,)=(A^{\widehat u})^{-1}(\widehat \rho\,)\in\mathbb T$ by Proposition~\ref{prop3.2}.

\smallskip
It follows that we can replace 
then the pair $(\tau^*, u^*(\cdot))$ by
a pair $(\widehat \tau, \widehat u(\cdot))$, for arbitrary $\widehat u(\cdot)\in\U$ and 
$\,\widehat \tau=(A^{\widehat u})^{-1}(\widehat \rho\,)$, with $\widehat \rho$ the optimal stopping time in 
\eqref{4.10}. 

\smallskip
However, the choice $(\tau^*,u^*(\cdot))$ leads to the  {\it ``fastest" 
termination time} possible, in the sense that the stochastic dominance relation
\begin{equation}
\label{4.12}
\P^*(\tau^*>t) \, \leq \, \P^{\, \widehat  u}\, \big(\widehat \tau>t\big), \qquad 0\leq t<\infty
\end{equation}
will hold for  {\it any} such pair $(\widehat \tau, \widehat u(\cdot))\in\mathcal T\times\U$,
as we saw in \eqref{4.11}.

\smallskip
It is noteworthy that the policy of ``full-bang control" $u^*(\cdot)\equiv 1,$ should lead to
a pair $(\tau^*, u^*(\cdot))$ with the optimality properties of both (\ref{4.12}) and of Theorem~\ref{thm4.2} (or  of Theorem 4.3),   
{\it despite} \,the presence of a (super)\,quadratic running cost of control. 

Therefore, in our  context, {\it bold play (``full-bang" control) is   optimal.} This is because it leads to a termination  time which is the earliest possible in the sense of \eqref{4.12}; {\it and the cost-reduction that early termination implies, outweighs the   cost of deploying ``full-bang" control. }

\section{Examples}
\label{sec6}

We present now a couple of examples, of distributions   for which   fairly explicit solutions are possible. 
These  are ramifications of examples discussed in our earlier work \cite{EKV}.

\subsection{The Gaussian Prior Distribution}
\label{gauss}

As a first simple illustration, let us consider the case of a \textsc{Gauss} prior distribution $\bm\mu$ with mean $m\in\R$ and variance $\sigma^2\in(0,\infty)$, i.e.,
\[
\bm\mu(\vd u)=\frac{1}{\sqrt{2\pi \sigma^2}}\, \exp\bigg\{-\frac{(u-m)^2}{2\sigma^2} \bigg\}\, \vd u\,.
\]

We are here, 
in other words, in the very special case of the \textsc{Kalman-Bucy} filter, where the posterior variance of the unobservable drift has deterministic evolution (modulo normalization of the quadratic variation). The time change $  A^u (\cdot)$ is thus natural  and canonical; it summarizes fully the impact of the observations filtration, and allows the separation of stopping decisions from observation costs.

In the present context, the functions $\, F$, $G$, $H$ and $\Psi$ take the very explicit form
$$
F(\theta,\zeta)\,=\, \frac{1}{\sqrt{1+ \sigma^2 \theta}}\, \exp\bigg\{ \,- \frac{1}{2\sigma^2} \, \bigg( \frac{(m + \sigma^2 \zeta)^2}{1+\sigma^2 \theta}- m^2\bigg) \bigg\}
$$
  \begin{equation}
\label{E:GH}
G(\theta,\zeta) \,= \,\frac{m+\sigma^2\zeta}{1+\sigma^2 \theta}\,, \qquad 
H(\theta,\zeta)\,=\, \Psi(\theta,x)\,=\,\frac{\sigma^2}{1+\sigma^2 \theta}\,=\,: {\bm \xi} (\theta)\,.
  \end{equation}
We fix also a continuous and non-decreasing function $h:(0,1]\to(0,\infty)$, such that there exists a number $u_0\in(0,1]$ for which \eqref{2.8.5} holds.

\smallskip
Now define a pair $\,(\tau^o,u^{o}(\cdot))\in\mathcal T\times \U\,$ by setting $u^{o}(\cdot)\equiv u_0$, and taking $\tau^o$ to be the optimal stopping rule for the problem of minimizing
over $\tau\in\mathcal T$ the expression
\begin{eqnarray*}
\E^{\P^{u^{o}}}\left[\int_0^{\tau}\left( c-\Psi^2 \big(A^{u^{o}}(t),\widehat X^{u^{o}}(t) \big) \right)\ud A^{u^{o}}(t)\right] &=& u_0^2  \cdot \E^{\P^{u^{o}}}\left[\int_0^{\tau} \Big(c-{\bm \xi}^2 \big(u_0^2t\big)\Big)\,\ud t\right] 
\end{eqnarray*}
with $c:=h(u_0)/u_0^2$ and the notation in \eqref{E:GH}. Clearly, the integrand 
$\,
c-{\bm \xi}^2 \big(u_0^2t\big)
\,$ 
is negative for $t\in[0,t_0)$ and positive for $t\in[t_0,\infty)$, where 
\[t_0=\frac{1}{u_0^2}\left(\frac{1}{\sqrt c}-\frac{1}{\sigma^2}\right)^+.\]
Consequently, we have $\tau^o=t_0\,$; on the strength of 
Theorem~\ref{thm4.3}, the pair $(\tau^o,u^{o}(\cdot))\in\mathcal T\times \U$ is then a minimizer for the problem in subsection~\ref{sec2.2}.

Also note that, in line with Remark~\ref{rem4.4}, we have $$\tau^o=\frac{1}{u_0^2}\,\tau^*\,,\qquad \text{where} \qquad  \tau^*=\left(\frac{1}{\sqrt c}-\frac{1}{\sigma^2}\right)^+$$ is the smallest optimal stopping time for the problem of least-squares estimation of $X$ given observations $Xt+W(t)$  and with cost of observation $c$ per unit of time, as studied in \cite{EKV}.

Note that $\tau_* =0$, i.e., that it is optimal not to take any observations at all, if observation ``costs too  much", i.e., if $\, c \ge \sigma^4\,.$

\subsection{The Bernoulli Prior Distribution}
\label{Bernoulli}

As a second example, let us consider the \textsc{Bernoulli} prior distribution 
\[\bm\mu\,=\,(1-p)\, \bm \delta_{-\beta}+p \, \bm \delta_{\beta}\]
with symmetric support,
where $p\in(0,1)$ and $\beta\in(0,\infty)$. 
Then
$$
G(\theta, \zeta) \,=\, \beta \, \frac{\, p \, e^{\beta \zeta}- (1-p) \, e^{-\beta y\zeta} \, }{\, p \, e^{\beta \zeta} + (1-p) \, e^{-\beta \zeta} \, }\,, \qquad  \qquad H(\theta,\zeta) \,=\, \beta^2 - G^2 (\theta,\zeta)
$$
and 
\[\Psi(t,x)\,=\,\beta^2-x^2 \,=\,: {\bm \psi} (x)\,.\]
We are here at the opposite extreme, vis-\`a-vis the example in subsection~\ref{gauss}: all these are functions 
of only   the spatial variable. As above, let $h:(0,1]\to(0,\infty)$ be a continuous and non-decreasing function such that there exists a number $u_0\in(0,1]$ for which \eqref{2.8.5} holds,
and let $c:=h(u_0)/u_0^2$.

By Theorem~\ref{thm4.3}, we define a pair $(\tau^o,u^{o}(\cdot))\in\mathcal T\times \U\,$ by setting $u^{o}(\cdot)\equiv u_0$, and take $\tau^o$ to be  the optimal stopping rule for the problem of minimizing
over $\tau\in\mathcal T$ the expression
\begin{eqnarray*}
\E^{\P^{u^{o}}}\left[\int_0^{\tau}\left( c-\Psi^2 \big(A^{u^{o}}(t),\widehat X^{u^{o}}(t) \big) \right)\ud A^{u^{o}}(t)\right] &=& u_0^2\cdot \E^{\P^{u^{o}}}\left[\int_0^{\tau}\left( c-  {\bm \psi}^2 \big(\widehat X^{u^{o}}(t) \big)  \right) \ud t\right],
\end{eqnarray*}
where $\widehat X:=\widehat X^{u^o}$ 
is a time-homogeneous diffusion in natural scale on the interval ${\cal I} = (- \beta, \beta)$, satisfying (cf.\,\eqref{3.19}):
\begin{equation}
\left\{\begin{array}{ll}
\ud \widehat X(t) = {\bm \psi}  \big(\widehat X(t)\big) \,u_0\,\ud N^{u^o}(t)\\
\widehat X(0)=x:=\beta(2p-1)\,.\end{array}\right.
\end{equation}
 Following the arguments of \cite[Section 4]{EKV},   where a Markovian embedding of the above stopping problem is carried out, and concentrating on the case $u_0=1$ for concreteness, it can be shown that 
$\tau^o \equiv 0$ when $\beta^4 \le c\,$ (again,  no observations are obtained at all, if their cost is too high); 
and that otherwise, $\tau^o $   has the form
\[
\tau^o=\inf \big\{t\geq 0:\widehat X(t)\notin(-a,a)\big\}
\]
   for some appropriate constant $a\in \big( \sqrt{ \beta^2 - \sqrt{c\,}\,} ,\beta \big)\,;$ in fact,   the   unique  solution of  the equation
   \begin{equation}
\label{E:FPP}
\int_0^a \frac{\,\ud \xi 
\,}{\,{\bm \psi} ^2 (\xi)\,}\,  \,=\,\frac{\,a\,}{c}\,.
\end{equation}
   
   \medskip
\section{Acknowledgments}

\smallskip
We are greatly indebted to Dr.\,V\'aclav E.  \textsc{Bene\v s} for  formulating, and suggesting to us, this and several related problems. We thank  Dr.\,Donghan \textsc{Kim} for his careful reading of the manuscript and   his comments; the participants at the {\it ``One World Optimal Stopping and Related Topics"} Seminar for their incisive observations and for pointing out relevant literature; and the referees, for their careful reading of our work and for their many and extremely valuable suggestions, which helped us improve the paper very significantly.

%
%
%

 

\bigskip


\begin{thebibliography}{99}


 \bibitem{BL}
 \textsc{Bensoussan, A. \& Lions, J.L.} (1982) { \em Applications of Variational Inequalities in Stochastic Control.}   North-Holland, Amsterdam and New York. 
%

 \bibitem{Ch} 
  \textsc{Chung, K.L.} (1974) {\em A Course in Probability Theory.}  Second Edition.  Probability and Mathematical Statistics: A Series of Monographs and Textbooks, Volume 21. Academic Press, New York.
  
  %
%
  \bibitem{DS}
 \textsc{Dalang, R.C. \& Shiryaev, A.N.} (2015) A quickest detection problem with observation cost. { \em Ann. Appl. Probab.} {\bf 25}, 1475-1512.   
 
   \bibitem{DZ}
 \textsc{Davis, M.H.A. \& Zervos, M.} (1994) A problem of singular stochastic control with discretionary stopping. {\it Ann. Appl. Probab.} {\bf 4}, 226–240.
 
  \bibitem{DS1}
 \textsc{Dubins, L.E. \& Savage, L.J.} (1965) {\it How to Gamble if You Must: Inequalities for Stochastic Processes.} McGraw-Hill Publishing Co., NY. Re-issued in 2014, edited and updated by W.D. Sudderth and D. Gilat, as a Dover Publication, Mineola, NY.
 
 

 
 
\bibitem{EKV}
\textsc{Ekstr\"om, E., Karatzas, I. \& Vaicenavicius, J.}  (2022)
Bayesian sequential least-squares estimation for the drift of a Wiener process. 
{\it Stochastic Process. Appl.} {\bf 145}, 335-352.

\bibitem{ELO}
\textsc{Ekstr\"om, E., Lindensj\"o, K. \& Olofsson, M.} (2022)
How to detect a salami slicer: a stochastic controller-and-stopper game with unknown competition. 
{\it SIAM J. Control Optim}. {\bf 60}, no. 1, 545-574.

\bibitem{EV}
\textsc{Ekstr\"om, E. \& Vaicenavicius, J.}  (2016)
Optimal liquidation of an asset under drift uncertainty. {\it SIAM J. Financial Math.} {\bf 7}, no. 1, 357-381.

\bibitem{ElK}
 \textsc{El\,Karoui, N.}  (1981)
 Les Aspects Probabilistes du Contr\^ole Stochastique. {\it Lecture Notes in Mathematics} {\bf 876},  73-238.



\bibitem{FS}
\textsc{Fleming, W.H. \& Soner, H.M.}  (2006) {\it Controlled Markov Processes and Viscosity Solutions.} 
Second Edition,  Springer-Verlag, New York. 


\bibitem{F1}
\textsc{F\"ollmer, H.} (1972) The exit measure of a supermartingale. {\it Z. Wahrscheinlichkeitstheorie \& Verw. Gebiete} {\bf 21}, 154-166.

\bibitem{HS}
\textsc{Harrison, J.M. \& Sunar, N.}  (2015) Investment timing with incomplete information and multiple means of learning. {\it Operations Research} {\bf 63}, 442-457.
 

%
%
%
%
%
%

\bibitem{KM}
\textsc{Kamizono, K. \& Morimoto, H.} 
  (2002). On a combined control and stopping time game. {\it Stochastics} {\bf 73}, 99–123. 
  
 \bibitem{K}
\textsc{Karatzas, I.}  (2003)  A note on Bayesian sequential detection with ‘expected miss’ criterion.  {\it Statistics and Decisions} {\bf 21}, 3-13.

\bibitem{KOWZ}
\textsc{Karatzas, I.,  Ocone, D.,   Wang, H. \&   Zervos, M.} (2000) Finite-fuel singular control with discretionary stopping. {\it Stochastics} {\bf 71}, 1-50.


 \bibitem{KS1} 
  \textsc{Karatzas, I. \& Shreve, S.E.} (1991) {\em Brownian Motion and Stochastic Calculus.}  Second Edition,  Graduate Texts in Mathematics, Volume 113. Springer-Verlag, New York.
  
  %
%

  
  \bibitem{KS3} 
\textsc{Karatzas, I.  \&   Sudderth, W.D.} (1999) Control and stopping of a diffusion process on an interval. {\it Ann. Appl. Probab.} {\bf 9}, 188-196.
  
 
  
  
  \bibitem{KS4} 
\textsc{Karatzas, I.  \&   Sudderth, W.D.} (2001) The controller-and-stopper game for a linear diffusion. {\it Ann. Probab.} {\bf 29}, 1111-1127.



  \bibitem{KW} 
\textsc{Karatzas, I.  \&   Wang, H.} (2001) Utility maximization with discretionary stopping. {\it SIAM J.   Control \&  Optim.} {\bf 39}, 306-329.


  \bibitem{KZ} 
\textsc{Karatzas, I.  \& Zamfirescu, M.} 
  (2006) Martingale approach to stochastic control with discretionary stopping. {\it Appl. Math. \& Optim.} {\bf 53}, 163-184.
  
   \bibitem{Kry} 
  \textsc{Krylov, N.V.} (1980) {\em Controlled Diffusion Processes.}  Springer-Verlag, New York.
  
  \bibitem{L} 
\textsc{Lepeltier, J.P.} (1985) On a general zero-sum stochastic control game with stopping strategy for one player and continuous strategy for the other. {\it Probab. \& Math. Statist.} {\bf 6}, 43-50.



  \bibitem{M} 
\textsc{Morimoto, H.} (2003) Variational inequalities for combined control and stopping. {\it SIAM J.   Control  \& Optim.} {\bf 42}, 686-708.
  
  %
  %
%
%
  \bibitem{PS}
 \textsc{Peskir, G. \& Shiryaev, A.N.} (2006) { \em Optimal Stopping and Free Boundary Problems.}   Birkh\"auser-Verlag, Boston. 
%
%
%
%
%
%
%
%
%
%
%
%
\bibitem{W} 
\textsc{Widder, D.V.} (1944) Positive temperatures on an infinite rod. {\it Trans. Amer. Math. Soc.} {\bf  75}, 510-525.

\end{thebibliography}
\end{document}